\documentclass[11pt]{article}
\usepackage{amsmath,amssymb,amsthm}
\swapnumbers \theoremstyle{plain}
\newtheorem{teor}{Theorem}[section]
\newtheorem{corol}[teor]{Corollary}
\newtheorem{lema}[teor]{Lema}
\newtheorem{propo}[teor]{Proposition}
\theoremstyle{definition}
\newtheorem{defin}[teor]{Definition}
\newtheorem{ejem}[teor]{Example}
\newtheorem{nota}[teor]{Remark}
\theoremstyle{remark}
\newcommand{\pa}{\raisebox{0.8mm}{\mbox{\Large $\wp$}}}

\newcommand{\F}{$\mathcal{F}$}

\usepackage{fancyhdr}
\title{Construction of topologies, filters and uniformities \\
for a product of sets }
\author{\sc Gustavo N. Rubiano O.\\ \small Mathematics Department \\ \small Universidad Nacional de Colombia, Bogotá}

\begin{document}
\maketitle
\begin{abstract}
Starting from filters over the set of indices, we introduce structures in a product of sets where
the coordinate sets have the given structures.
\end{abstract}
\fancyhead{}
 \fancyhead[CE]{Gustavo N. Rubiano O.}
\fancyhead[CO]{{Construction Of Topologies, Filters, And Uniformities $\ldots$}}

\section{Introduction and historical note}
 H. Tietze [Ti] defined in 1923 a topology for the cartesian product of topological spaces; this topology became
 known in the literature as the ``box topology". The accepted definition for the topology in
 the product of spaces was introduced by A. Tychonoff [Ty] in 1927; although it coincides with the Tietze
 topology for finite products, in general it is different.

It seems that one of the reasons why Tychonoff's definition became
\emph{the definition} (although it goes against common sense to
take only finite open sets in products with index with any
cardinality) is the fact that with this definition the well known
Tychonoff Theorem is satisfied. (Tychonoff gave his definition so
that compactness would be expressively satisfied as a product
property.)  Many other desirable properties are exclusive for the
Tychonoff's topology, such as being the best topology for which
the projection functions are continuous, o have a minimax
property. With such a property Tychonoff's topology for a product
of Hausdorff spaces is maximal for compactness and minimal for
Hausdorff (Proposition 2.10).

We study in this article a definition for topologies over a
Cartesian product of sets introduced by C. Knight in 1964 [Kn], a
definition that generalizes the above mentioned definitions and it
is known as the \F-topology, (the reason for this notation is the
fact that such topology is based on a filter \F over the set of
indexes involved in the indexing of the family of topological
spaces used to produce the product set). Incidentally, this
definition allowed us to introduce a mechanism for similar
constructions on other structures such as filters and
uniformities.

\section{Topologies for a product of topological spaces}
\subsection{Notation and basic concepts}

The following well-known definitions and facts are included so that the notation that we will use
is established.  Given a set $X$, the collection of all the subsets is denoted either as $\pa(X)$
 or as $2^X$ and we call it parts of $X$.  By   $\{X_i\}_{i\in I}$        we mean a family of sets indexed by a set
$I$. Its Cartesian product is by definition the set
$$\prod_{i\in I}X_i  :=  \{ f:  I \longrightarrow
\bigcup_{i\in I} X_i \mid  f( i ) \in X_i \}.$$
 A \emph{topological structure} for a set $X$ is a family ${\cal G}$ of subsets of X closed for arbitrary unions as
well as for finite intersections; The elements of ${\cal G}$ are
called open sets and the pair $(X, {\cal G})$  is by definition a
topological space. We briefly write $X$ when it is not necessary
to specify ${\cal G}$.  Given a topological space $(X, {\cal G})$
we say that    $V \subseteq X$            is a neighborhood of $x
\in X$
---we denote $V_x$---- if and only if there is a $U \in {\cal G}$ such that $ x \in U \subseteq V_x$.
The set of all the neighborhoods of $x$ is denoted $\mathcal V(x)$. We say that $X$ is Hausdorff or
$T_2$ if for every pair of points $ x,y \in X$ there are $V_x , V_y$  with  $ V_x\cap V_y =
\varnothing $.

 A family $ {\cal B} \subseteq \pa(X)$    is \emph{base} of a unique topology for $X $   if and only if it is true
  that
 $X= \bigcup\{B \mid B \in {\cal B}\}.$
and Given any $ U,V \in {\cal B}$   and  $ x \in U\cap V$ there is
$B$ in ${\cal B}$ such that $ x \in B \subseteq{ U \cap V}$. That
is, $ U \cap V$ is a union of elements of ${\cal B}$ for every
pair  $ U,V$ of ${\cal B}$.

Let's construct now the topology ${\cal G}$ for which ${\cal B}$
is a base. We define${\cal G}$ saying that $U\in {\cal G}$ if and
only if $U $ is a union of elements of $ {\cal B}$. This last
topology is known as the topology generated by the base $ {\cal
B}$ and we denote it  ${\cal G} = \langle {\cal B} \rangle$.

 A function $f$ between topological spaces $f: (X,{\mathcal G}) \longrightarrow (Y,{\mathcal H})$
 is said to be continuous if
and only if for each $V \in {\mathcal H}$ we have that  $f^{-1}(V) \in {\mathcal G}$.

\bigskip
A \emph{filter structure} for a set $X$ is a family \F \ of
non-empty subsets of $X$ closed for finite intersections as well
as for supersets; that is, $ {{\cal F}} = \{ F_j \, \mid \, j \in
J \}$  with $\varnothing \neq F_j \subseteq X $  such that
\begin{enumerate}
\item If $F_1,F_2 \in {\mathcal F}$   then  $F_1\cap F_2 \in\mathcal
F$,
\item If $F \in {\mathcal F}$   and $F \subseteq G$ then $G \in {\mathcal
F}$.
\end{enumerate}
 We can
rewrite these two properties as
\begin{enumerate}
\item[3.]  $F_1\cap F_2 \in\mathcal F$ if and only if  $ F_1 \in {\mathcal F}$ y $F_2 \in\mathcal F$.
\end{enumerate}
 A family  $\mathcal B\subseteq 2^X$ is a\emph{ base of filter} for a unique
filter ${\mathcal F}$ of $X$  if and only if it satisfies,
$\varnothing \notin \mathcal B\ne\varnothing $, and  Given
$B_1,B_2 \in\mathcal B$ there is $B_3\in \mathcal B$ such that
$B_3\subseteq B_1\cap B_2$  (this definition is not punctual, as
it is the case in the definition of a base of a topology).
 The generated filter \F \ is defined by $F \in\mathcal F$ if and only if there is$B \in \mathcal B$ such that
 $B\subseteq F$; we denote it  ${\mathcal F} = \langle \mathcal B \rangle$ ---the collection of supersets
 of the elements of B---.

 Given $A\subseteq X$ the family $\mathcal B = \{ A \}$ is a filter
base. The filter ${\mathcal F}_{\langle A \rangle} :=  \langle {A}
\, \rangle$ associated with this base is called the principal
filter associated with $A$.

\medskip

 If in the definition of filter it is allowed that, $\varnothing \in \mathcal{F}$
 we will always obtain the same trivial filter $\pa(X)$ of the parts of $X$; this is the reason why it is
the custom not to admit the empty set in the definition, but to
allow to accept $2^X$ as a filter called \emph{trivial}.  A filter
${\cal G}$ in X is called an ultrafilter if and only if given any
$A\subseteq X$ then  $A\in \mathcal G$ or $A^c \in \mathcal G$.

Let $f: X\longrightarrow Y$ be a function between sets. Given a
filter \F \ in $X$ the collection
$$f({\mathcal F}) := \{f(F)\mid F\in {\mathcal F} \},$$
is a base for a filter in $Y$ denoted  $\langle f({\mathcal F}) \rangle$. If $f$ is onto, then
$f({\mathcal F})$ is automatically a filter, and if ${\mathcal F}$ is an ultrafilter, so is
$f({\mathcal F})$.
\subsection{The construction}

Given a family   $\{(X_i, \mathcal{J}_i)\}_{i \in I}$ of topological spaces, we define a box $\bold
U$ in
 the set product $\prod_{i \in I} X_i$  as
\begin{equation*}
\bold U  := \langle U_i \rangle :=  \prod_{i\in I} U_i, \ U_i \in \mathcal{J}_i,\
 i \in I.
\end{equation*}
The set   $\delta(\bold U)$ of \emph{distinguished indexes} of a
box is defined as
\begin{equation*}
\delta(\bold U)=\delta \langle U_i \rangle := \{i \in I \mid U_i = X_i\}.
\end{equation*}
The \emph{support } $\sigma(\bold U)$ of the box  is defined as
\begin{equation*}
\sigma(\bold U)=\sigma \langle U_i \rangle := \{i \in I \mid U_i \ne X_i\}.
\end{equation*}

\begin{propo}Let $ \mathcal{F}$ be a non-empty collection of subsets of $I$.
The boxes $\langle U_i \rangle$ such that $\delta \langle U_i \rangle \in  \mathcal{F}$ form a base
for a topology in $\prod_{i \in I} X_i$ if and only if \F \ is closed for finite intersections.
\end{propo}
\begin{proof}
 $(\Rightarrow)$ Let $A,B \in \mathcal{F}$  and $\langle A_i \rangle , \langle B_i \rangle $
 be two boxes such that $\delta \langle A_i \rangle= A$, $ \delta \langle B_i \rangle=B$. The box
$\langle A_i \rangle \cap \langle B_i \rangle  =  \langle A_i \cap B_i \rangle$ is such that
 $\delta \langle A_i \cap B_i \rangle \in \mathcal{F}$; but
 $$A\cap B = \delta \langle A_i \rangle \cap \ \delta
 \langle B_i \rangle = \delta \langle A_i \cap B_i \rangle \in \mathcal{F}.$$
$(\Leftarrow)$ Since \F$\ne \varnothing$, there is at least an open set. Let $\langle A_i \rangle ,
\langle B_i \rangle $ be boxes such that $\delta \langle A_i \rangle, \delta \langle B_i \rangle
\in $\F. Then
\begin{itemize}
\item [(i)] $\langle A_i \rangle \cap \langle B_i \rangle  =  \langle A_i \cap B_i \rangle$,
\item [(ii)] $\delta \langle A_i \cap B_i \rangle = \delta \langle A_i \rangle \cap  \delta
 \langle B_i \rangle \in \mathcal{F}$.\qed
\end{itemize}
\renewcommand{\qed}{}
\end{proof}
We denote the topology generated by this base

 \medskip
\centerline{$\prod_{i \in I}^{\mathcal F} X_i$, or simply
$\mathcal{J_{\mathcal{F}}}$,}
\medskip

\noindent and we call it the $\mathcal F$--topology for the
product of the spaces   $X_i$. An element of the base can then be
denoted as
\begin{equation*}
\prod_{i \in F^c}U_i \, \overset{{\mathcal F}}{\times} \, \prod_{i \in F} X_i.
\end{equation*}
\begin{nota} The above result is true even if the $X_i$ are not topological spaces. This is also
true, even if the elements of the boxes are not necessarily open sets; this would obviously mean
that we need to forget about the topological structures in the factor spaces $X_i$,  but also this
has the advantage that we can generalize the way we are going to do it in Sections 2 and 3.
\end{nota}
 The \emph{Tychonoff's topology} has as base the set of boxes    $\langle U_i \rangle$           such that $U_i = X_i$
  for every $i$ in $I$
except for a finite number of indexes $i$; That is,  $\langle U_i \rangle$      is in the base if
and only if
$$
\delta \langle U_i \rangle := \{i \in I \mid U_i = X_i\} \in  \mathcal{F}
$$
where $ \mathcal{F}$ is the Fr$\grave{e}$chet filter in $I$ ---the
filter of the cofinite sets over I; more precisely, $F\in
\mathcal{F}$ if and only if the complement of $F$ in $I$ is
finite-.

Since every filter is obviously a closed family for finite intersections, is clear then that the
Tychonoff's product topology is that one determined by the filter of the cofinite sets (we denote
by   $\bigotimes_{i\in I} \, X_i$  the product space with the Tychonoff's topology). The box
topology is generated from the trivial filter $\pa(I)$ ---parts of $I$---; for this box topology,
an arbitrary product of open sets is open, or in other words, every box is open (we denote
$\square_{i \in I} \, X_i$ the product space with the box topology).

Let's remember that for every non-empty set $X$, the collections
$Top(X)$  of the topologies on $X$ and  $Fil(X)$ of the filters on
$X$, are ordered in a natural way by inclusion among sets.
$Top(X)$ is a complete lattice with the trivial topology with only
two elements   ($\emptyset$ and $X$) as the smallest element, and
the discrete topology or parts of $X$ as the largest element.

Unless we say otherwise, the factor spaces $X_i$ considered in this article will have a non-trivial
topology (factors with trivial topologies do not contribute in any way to \F-topologies).

The following central proposition tells us that the construction
is much more than a bijection function from $Fil(I)$ to the
\F-topologies for the family $\{(X_i, \mathcal{J}_i)_{i \in I}\}$.
\begin{propo}  For a fixed collection  $\{(X_i, \mathcal{J}_i)\}_{i \in I}$ of non-trivial topological spaces, the
function
$$\mathcal{F} \mapsto \prod_{i \in I}^{\mathcal F} X_i $$
is an order immersion:

\medskip
\centerline{$ \mathcal{F}\le \mathcal{G} \Longleftrightarrow \prod_{i \in I}^{\mathcal F} X_i \le
\prod_{i \in I}^{\mathcal G} X_i.$}
\end{propo}
\begin{proof} $(\Rightarrow)$ If  $U \in \prod_{i \in I}^{\mathcal F} X_i$  is a basic open set
 observe that it is also a basic
open set in  $\prod_{i \in I}^{\mathcal G} X_i$ since if $F=\delta \langle U \rangle \in {\mathcal
F}$ then $F\in {\mathcal G}$ .

$(\Leftarrow)$ Suppose $\prod_{i \in I}^{\mathcal F} X_i \le \prod_{i \in I}^{\mathcal G} X_i$ it
must be that  ${\mathcal F}\le {\mathcal G}$. Otherwise, there is  $F\in {\mathcal F}$ with $F
\notin {\mathcal G}$ and therefore  $G\nsubseteq F$ for every   $G \in {\mathcal G}$, or,
equivalently   $F^c\nsubseteq G^c$ for every $G$, that is, for every ${ G}$ there is an index  $i_G
\in I$ with $i_G \in F^c$ and  $i_G\notin G^c$.

For this fixed $F$, the open set
\begin{equation}
\prod_{i \in F^c}U_i \overset{{\mathcal F}}{\times} \prod_{i \in F} X_i
\end{equation}
of the \F-topology, is also ---by hypothesis--- in the ${\mathcal
G}$-topology and therefore it must be a union of elements of the
base in the ${\mathcal G}$-topology. Let's see that this is not
true.

In the ${\mathcal G}$-topology the base is formed from sets of the
form
\begin{equation*}
\prod_{i \in G^c}U_i \overset{{\mathcal G}}{\times} \prod_{i \in G} X_i.
\end{equation*}
For every basic element ---determined by a  $G\in {\mathcal G}$---
\begin{equation*}
\prod_{i \in G^c}U_i \overset{{\mathcal G}}{\times} \prod_{i\in G} X_i
\end{equation*}
there is $i_G$  with $i_G \in F^c$  and  $i_G \in G$; for this
index $i_G$ we take and open set $U_{i_G}\ne X_{i_G}$ and an
element $x_{i_G} \notin U_{i_G} $ to construct a point  $x=
(\ldots,x_{i_G}, \ldots)$ ---the  $x_i$  with $i$ coordinates
different from  ${i_G}$ are arbitrary--- which does not belong to
the open set
\begin{equation*}
\prod_{i\in F^c}U_i \overset{{\mathcal F}}{\times} \prod_{i \in F}X_{i}
\end{equation*}
but it belongs to
\begin{equation*}
\prod_{i \in G^c}U_i \overset{{\mathcal G}}{\times} \prod_{i\in G} X_i.
\end{equation*}
Therefore, no element of the base is in
\begin{equation*}
\prod_{i \in F^c}U_i \overset{{\mathcal F}}{\times} \prod_{i \in G} X_i
\end{equation*}
and so (1) cannot be a union of elements of the base, which implies that it does not belong to the
${\mathcal G}$-topology, and therefore,
 ${\mathcal F}$-topology $\nsubseteq$ ${\mathcal G}$-topology
which contradicts the hypothesis.
\end{proof}
The  following example presents a particular class of
\F-topologies, which, as far as the author knows -search in
MathSci and Zentrakbkatt- are the only \F-topologies referred in
the literature, for example [Bo], [Kn], [Vi].

\begin{ejem}\emph{A special class of product spaces}. Let $I$ be a set with infinite
cardinality. Say $|I|=\mathfrak c$. If   $\mathfrak d$  is another cardinal with  $\mathfrak d \le
\mathfrak c$ then the set $\mathcal{F}=\{F \subseteq I : |X-F|< \mathfrak d \}$ of the complements
of the subsets of $I$  with cardinal less than   $\mathfrak d$ is a filter over $I$, called the
filter of the $\mathfrak d$-complements. The corresponding \F-topology has as basic open sets the
boxes  $\langle U_i \rangle$  such that  $|\sigma \langle U_i \rangle | < \mathfrak d$; when
$\mathfrak d = |\mathbb N|$, the corresponding filter is the one of the cofinite sets, and if
$\mathfrak d=|\mathbb R|$ then we have the filter of the complements of the countable sets.
\end{ejem}
\subsection{Hausdorff} When is the \F-product of Hausdorff spaces Hausdorff?

 Given two
points    $x=(x_i)$, $y=(y_i)$ in $X=\prod_{i\in I} X_i$  with $x\ne y$, there is at least an index
$i$ with $x_i \ne y_i$. For this index we can find separate neighborhoods in the $X_i$, and then
going into the product space we should find proper open sets in $X_i$. If we remember that the
basic open sets are of the form
\begin{equation*}
\prod_{i \in F^c}U_i \, \overset{{\mathcal F}}{\times} \, \prod_{i \in F} X_i \qquad (F\in
{\mathcal F}),
\end{equation*}
\noindent then to be able to separate them with disjoint
neighborhoods it must happen that  $i\notin F$ for some  $F \in
{\mathcal F}$.
\begin{propo}  Let  $\mathcal{F}$ be a filter in $I$. The following are equivalent propositions:
\begin{enumerate}
\item For each $i\in I$ there is $F\in \mathcal{F}$ with  $i\notin F$.
\item  For each  $i\in I$ the set   $I-\{i\}\in \mathcal{F}$.
\end{enumerate}
\end{propo}
\begin{proof}($\Rightarrow$) Given $i$ there is $F$    with  $i \notin F$, then $F \subseteq I-\{i\}$ and
so  $I-\{i\} \in \mathcal{F}$.

($\Leftarrow$)It is immediate.
\end{proof}
\begin{defin} A filter with the properties in the previous proposition is called
\emph{saturated}.
\end{defin}
\begin{itemize}
\item If a filter \F \ is saturated then the filter of the cofinite sets is contained in \F, since by
taking finite intersections in  \F \   we can produce any element of cofinite sets.
\item There are
saturated filters different from the filter of the cofinite sets, for example the filter of the
conumerable sets formed by all the subsets of I that have numerable complements finite of infinite
(see Example 2.4).
\item Every \emph{free ultrafilter} (ultrafilters not generated by an element) is a
saturated filter, since otherwise there would be an index with  $i\in I$ con $I-\{i\} \notin
{\mathcal F}$ and this would imply that $\{i\} \in {\mathcal F}$ and ${\mathcal F}$ wouldn't be
free.
\end{itemize}
\noindent What conditions on the filter \F \  assure us that the projection functions
$$p_i:\prod_{i\in I}^{\mathcal{F}} X_i \longrightarrow X_i$$
are continuous?
\begin{propo} Let  $\{(X_i, \mathcal{J}_i)\}_{i \in I}$ be a collection of non-trivial topological spaces and \F \  a filter
in $I$. The \F--topologies make the projection functions continuous if and only if the filter \F is
finer than the filter of the cofinite sets.
\end{propo}
\begin{proof}
$(\Rightarrow)$ Let's remember that the Tychonoff topology is characterized as the \emph{best}
topology over the product set that make the projection functions continuous ---best means the
topology of the least open sets, or the coarsest, or the least fine---. If the projection functions
are continuos then
$$cofinite sets\text{--topology} \, \subseteq
 \mathcal{F}\, \text{--topology}$$
and by Proposition 2.3 we have  $cofinite \ sets \subseteq
\mathcal{F}$.

\noindent $(\Leftarrow)$  According to Proposition 2.3 if
$cofinite \ sets \subseteq \mathcal{F}$ then
$$cofinite \ sets \text{--topology} \, \subseteq
 \mathcal{F}\, \text{--topology}$$
and therefore for the \F--topology the projection functions are also continuous.
\end{proof}
\begin{propo} Let  $\mathcal{F}$ be a saturated filter of $I$. The $\mathcal{F}$-topology is Hausdorff if and
only if each $X_i$  is Hausdorff.
\end{propo}
\begin{proof}($\Rightarrow$) We will show that each factor $X_i$   is homeomorphic to a
subspace   $X_i^y$ of  $X$, and since Hausdorff is a hereditary
and invariant property we have that $X_i$ is also Hausdorff. To
construct $X_i^y$ we take an arbitrary point \emph{y} of X and the
corresponding I; we define
\[X_i^y:=\{(x_j)\mid x_j=y_j \text{ para } j\neq i\}=\prod_{j\neq i}\{y_j\}\times X_i.\]
The restriction of the projection function
\[\left. p_i \right|_{X_i^y} :X_i^y \longrightarrow X_i\]
is a homeomorphism ---remember that \F  is saturated. \noindent ($\Leftarrow$) This implication is
precisely the reason for introducing the concept of saturated.
\end{proof}

The last proposition implies that, in general, the product of Hausdorff spaces is not Hausdorff; a
necessary condition is that the filter is saturated.

\sffamily
\begin{itemize}
\item [$\bullet$]
 Using techniques of non-standard analysis, Y. Suemura, Y. Nakano in [SN] pretends to proof that
given a filter $\mathcal F$, the product space  $\prod_{i \in
I}^{\mathcal F} X_i$ is Hausdorff if and only if each space is
Hausdorff. The ``proof" of this last result is incorrect for the
simple reason that the proposition is \emph{false}! as we have
establish it in our previous proposition and the following
counterexample.
\end{itemize}
\begin{ejem} \textsf{Let $I=\mathbb N$ and \F=$\langle 1 \rangle$ the (non-saturated) ultrafilter generated by
the element $1\in \mathbb N$. For each $i$ we define $X_i := (\{0,1\}, discrete)$ with which each
factor turns out to be Hausdorff. However, the product space $X=\prod_{i \in I}^{\mathcal F} X_i$
is not Hausdorff, since the points $(1,0,0,0, \ldots)$ and  $(0,0,0,0, \ldots)$ cannot be separated
in the \F-topology considering that each basic open set in the index 1 takes the whole space
$X_1$.}
\end{ejem}

\rmfamily

 \subsubsection{Hausdorff and compactness}
A topology $(X,\mathcal{J})$ that is Hausdorff and compact is characterized by a \emph{mini-max}
property [He]; exactly:
\begin{propo}\label{minimax} If  $(X,\mathcal{J})$ is Hausdorff and compact then  $\mathcal{J}$ is:

 \begin{itemize}
\item \emph{Hausdorff minimal }: $ \mathcal{H}\varsubsetneq \mathcal{J}$   implies that
 $\mathcal{H}$ is not Hausdorff.
 \item \emph{compact maximal }: $ \mathcal{J}\varsubsetneq \mathcal{H}$   implies that $\mathcal{H}$  is not
compact.
\end{itemize}
\end{propo}
\begin{proof} Let's remember that: if  $(X,\cal G)$, $(Y,\cal H)$ are topological spaces with $X$ compact and $Y$ a
Hausdorff space, then a continuous bijection   $f: X
\longrightarrow  Y$ is a homeomorphism. Therefore, considering
---in both directions---  the identity function  $id_X:
(X,\mathcal{J}) \longrightarrow  (X,\mathcal{H})$  we necessarily
have   $ \mathcal{J}= \mathcal{H}$.
\end{proof}

\begin{corol} If   $(X,\mathcal{J})$ is Hausdorff and compact then for every space $(X,\mathcal{H})$  Hausdorff and
compact that is compared with $(X,\mathcal{J})$ we have  $\mathcal{H} = \mathcal{J}$.
\end{corol}
 From this, if a topology for the product of Hausdorff and compact spaces is strictly contained in
the Tychonoff topology ---which makes it compact--- for this
topology the product of Hausdorff spaces wouldn't be Hausdorff. If
the topology contains the Tychonoff topology, the product of
compact spaces is not necessarily compact (so there is no
Tychonoff theorem for such topology). Because of this, the
Tychonoff topology can be considered as \emph{mini--max} in the
category of bicompact spaces ---Hausdorff and compact spaces---
and in particular    minimal Hausdorff or Hausdorff minimal.  In
1940, Katetov [Ka] characterized the Hausdorff spaces that are
minimal Hausdorff and he produced an example of a minimal
Hausdorff space that is not compact.

\begin{propo} Let \F \  be a saturated filter ---so that we can compare. If
$\prod_{i \in I}^{\mathcal F} X_i$ is Hausdorff and compact then
\F  \ is cofinite (the filter of cofinite sets).
\end{propo}
\begin{proof} Since \F \ is saturated
we have
$$\prod_{i \in I}^{\mathcal Cofinites} X_i \subseteq \prod_{i \in
I}^{\mathcal F} X_i.$$
 On the other hand, if  $\prod_{i \in I}^{\mathcal F} X_i$ is Hausdorff and compact, then each $X_i$  is Hausdorff
and each $X_i$  is compact since the projection functions are continuous and onto functions;
therefore,   $\prod_{i \in I}^{\mathcal Cofinitos} X_i$ is also Hausdorff and compact. This last
fact implies
  $$ \prod_{i \in I}^{\mathcal Cofinites} X_i =  \prod_{i \in I}^{\mathcal F} X_i$$
and therefore $\mathcal{F}$ is cofinite \ sets.
\end{proof}
\begin{corol} Let $\mathcal{U}$ be a ultrafilter free on $I$. If every
factor is Hausdorff then $\prod_{i \in I}^{\mathcal U} X_i$ is not a compact space even if each
factor space is.
\end{corol}
\begin{proof} The free ultrafilters contain the filter of the cofinite \ sets.
\end{proof}
\begin{nota}  The \emph{generalized Tychonoff theorem} (compactness in the product for
topologies different from the Tychonoff topology) is not valid in
the category of the Hausdorff spaces when the \F -topology comes
from a free ultrafilter. Even more, it only makes sense over
filters non comparable with the filter of the cofinite \ sets;
this, radically excludes the saturated filters different from the
cofinite \ sets.

And since each filter accepts to be contained in a free ultrafilter:
\begin{itemize}
\item \emph{The generalized Tychonoff theorem} is not valid for an ample range of \F-topologies.
\item Let \F \ be a saturated filter. The generalized Tychonoff theorem ---for the \F-topologies--- in
the Hausdorff category is true if and only if \F \ is the filter
of the cofinite \ sets.
\end{itemize}
\end{nota}

With respect to other separation properties weaker than Hausdorff we have the following.
\begin{lema} If  $\{X_i\}_{i\in I}$ is a family of spaces $T_1$ with $|X_i|\ge 2$
and $|I|\ge\aleph_0=|\mathbb N|$ then  $X=\square_{i \in I} \, X_i$ (the box topology) is not
compact.
\end{lema}
\begin{proof} Let' observe that $X$ contains a closed subspace, discrete and infinite, which contradicts that
$X$ be compact since every closed subset of a compact space is compact as subspace.  For each $i\in
I$ let $D_i:=\{x_i^0,x_i^1\}\subseteq X_i$ with the discrete topology. The space
$$D=X=\square_{i \in I} \, D_i,$$
as subset of $X$ is closed, since it is product of closed factors; it is also discrete, since each
point  $x=(x_i)_{i\in I} \in\square_{i\in I} \, D_i$ is open, considering that in each coordinate
$i$ of $x$ we take the other point in $D_i$  and obtain
$$x= (x_i) = D\cap (\prod_{i}\{x_i^0\}^c )$$
where the second set of the intersection is an open set. $D$, of course, is infinite.
\end{proof}
\begin{propo} Let $\{X_i\}_{i\in I}$  be a family of spaces $T_1$  with $|X_i|\ge2$ and  $|I|\ge\aleph_0=|\mathbb N|$.
If \F \ is a saturated filter different from the filter of the
cofinite \ sets then the \F-topology is not compact.
\end{propo}
\begin{proof} Let  $F\in \mathcal{F}$ with  $F^c$  be an infinite set; for each $i\in F$  take $a_i\in X_i$.
 The set $A$ defined as   \begin{equation}
A=:\{ x=(x_i)\in \prod_{i\in I}X_i : x_i = a_i,\  i\in F\}
 \end{equation}
is closed since
$$A=\prod_{i\in F}\{a_i\} \times \prod_{i\in F^c}X_i = \bigcap_{i\in F}p_i^{-1}\{a_i\}.$$
On the other hand $A$ is isomorphic with the space $\square_{i \notin F} \, X_i$ since if
$\mathcal{F}_A$ is the topology of subspaces of  $\prod_{i \in I}^{\mathcal F} X_i$ for $A$ then
 $$U \in \mathcal{F}_A \Leftrightarrow U = A \cap
 (\prod_{i\in F}X_i \times \prod_{i\in F^c}U_i)=\prod_{i\in F}\{a_i\} \times
 \prod_{i\in F^c}U_i.$$
Therefore $A$ is closed non-compact and then neither can the \F-topology be compact.
\end{proof}
\section{Resolvability}
The concept of resolvable (a space $X$ is resolvable if it contains two subsets $E, D$ dense and
disjoint such that $X=D\cup
 E$) was introduced and studied by E. Hewitt [He] in his famous article from
1943, among other reason with the purpose of building inside the complete lattice $Top(X)$ finer
topologies (or expansions) than one given.

 Since then, more general definitions concerning
resolvability have been studied. For   $0\le n\le \omega$, a space
$X$ is $n$--resolvable if $X$ has $n$ disjoint dense subsets
(being $n$--resolvable for each $n$ implies being
$\omega$--resolvable \cite{illanes})); even more, the last
definition can be extended to infinite cardinals. We will show in
this section an ample class of resolvable spaces.

In 1947 Katetov [Ka 2] retook the resolvability to answer a question of existence: is there a
topological space $X$ without isolated points for which every real valued function defined on $X$
has a point where the function is continuous?

If a space $X$ has two subsets $E, D$ dense and disjoint such that $X=D\cup
 E$ then the function $f$ defined as $f:X \longrightarrow \mathbb R$ where $f(E)=0$
 and $f(E) = 1$ cannot be continuous at any point. Then an affirmative answer to the
Katetov's problem must be in the class of the non-resolvable or irresolvable spaces (spaces that do
not posses dense disjoint subsets).

Let  $\{X_i\}_{i\in I}$  be a family of topological spaces and \F
\  a filter in $I$. Given a fixed point $x=(x_i)$ in  the
Cartesian product  $X=\prod_{i \in I} X_i$, we define the
following subset called the \emph{equalizer} by  \F \  for the
point $x$,
 $$\sum^{\mathcal{F}}(x)=\{(z_i)\in X \mid \{i\in I: x_i = z_i  \}\in \mathcal{F}\}=\bigcup_{F\in
\mathcal{F}}\left(\prod_{i \in F} \{x_i\} \times \prod_{i \in F^c} X_i\right)$$ (this definition
generalizes the definition given by the equation (I.2)).

\begin{propo} Let $\{X_i\}_{i\in I}$  be a family of topological spaces and \F \ a filter in $I$. For each $x\in X$
the set $\sum^{\mathcal{F}}(x)$
 is dense in $ X=\prod_{i \in I}^{\mathcal F} X_i$.

  Also, if $x,y \in X$ are such that  $\{i:x_i \neq y_i\}\in \mathcal{F}$ ---we call them \emph{different by} \F ---then
$$\sum^{\mathcal{F}}(x)\bigcap\sum^{\mathcal{F}}(y)=\emptyset.$$
\end{propo}
\begin{proof} Let
$$U=
\prod_{i \in F^c}U_i \overset{{\mathcal F}}{\times} \prod_{i \in F} X_i$$ be a basic open set and
 $y=(y_i)\in U$. For the index set  subset
$$\delta( U)=\delta \langle U_i \rangle := \{i \in I \mid U_i = X_i\}\in  \mathcal{F}$$
we define the point $z=(z_i)_i\in I$ as
\[z_i=
\begin{cases}
x_i & \text{if $i\in \delta(U)$}, \\
y_i & \text{if $i\notin \delta(U)$}.
\end{cases}
\]
The point   $z\in U\cap \sum^{\mathcal{F}}(x)$, which proves density.

 To see that the sets are
disjoint, let  $z\in \sum^{\mathcal{F}}(x)\cap\sum^{\mathcal{F}}(y)$ ---$z$ is equal to $x$ and $y$
by means of \F ---therefore   $F_x =\{i\in I: x_i = z_i  \} \in \mathcal{F}$   and $F_y =\{i\in I:
y_i = z_i  \} \in \mathcal{F}$, hence $F_x \cap F_y \in \mathcal{F}$ but
 $$F_x \cap F_y \subseteq \{i:x_i = y_i\}= \{i:x_i \ne y_i\}^c$$
and then   $\{i:x_i \ne y_i\}^c\in \mathcal{F}$ and this
contradicts that $\mathcal{F}$  is a filter.
\end{proof}
 Since each set that contains a dense subset is also dense, this last proposition shows that
 $ X=\prod_{i \in I}^{\mathcal F} X_i$
is a solvable space.
\section{ Filters for a product of sets}
\subsection{The construction}

As we noted in Observation 2.2, the construction is valid even if the family $\mathcal{J}_i$
considered in each $X_i$ is not a topology.

Let's consider then the case of a family  $\{ X_i\}_{i\in I }$ of non-empty sets and for each index
$i$ a filter $\mathcal{F}_i$  in $X_i$. In other words, we consider a family of pairs $\{(X_i,
\mathcal{F}_i)\}_{i \in I}$  and we call each pair a \emph{filtered space} (imitating the
topological space concept in which the topological structure is changed by the filter structure).

Given the cartesian product $X=\prod_{i \in I} X_i$ of the sets
$X_i$, we ask: How to construct a filter in $X$ from the filter
family $\{\mathcal{F}_i\}_{i \in I}$ in the factors? The answer is
the  following construction.

Given a family  $\{(X_i, \mathcal{F}_i)\}_{i \in I}$  of filtered spaces, we define a \emph{box}
$\bold F$ in the set product  $\prod_{i \in I} X_i$ as
\begin{equation*}
\bold F  := \langle F_i \rangle :=  \prod_{i\in I} F_i, \ F_i \in \mathcal{F}_i,\
 i \in I.
\end{equation*}

 The set  of
distinguished indexes of a box is defined as
\begin{equation*}
\delta(\bold F)=\delta \langle F_i \rangle := \{i \in I \mid F_i = X_i\}.
\end{equation*}
The \emph{support } $\sigma\langle F_i \rangle$  of the box  $\langle F_i \rangle$ is defined as

\begin{equation*}
\sigma(\bold F)=\sigma \langle F_i \rangle := \{i \in I \mid F_i \ne X_i\}.
\end{equation*}
\begin{propo} Let  $ \mathcal{F}$  be a filter on $I$. The boxes  $\langle F_i \rangle$ such that
$\delta \langle F_i \rangle \in \mathcal{F}$
 form a base for a filter
in  $\prod_{i \in I} X_i$.
\end{propo}
\begin{proof} $(\Rightarrow)$ Let   $\textbf{F}=\langle F_i \rangle$, $\textbf{G}= \langle G_i \rangle $  be two boxes such that
$\delta \langle F_i \rangle \in \mathcal{F}$ and $ \delta \langle G_i \rangle \in \mathcal{F}$. The
box $\langle F_i \rangle \cap \langle G_i \rangle  =  \langle F_i \cap G_i \rangle$ is such that
$\delta \langle F_i \cap G_i \rangle \in \mathcal{F}$.

$(\Leftarrow)$ Since  \F$\ne \varnothing$, there is at least one element in the filter. Let
$\langle F_i \rangle , \langle G_i \rangle $ be boxes such that  $\delta \langle F_i \rangle,
\delta \langle G_i \rangle \in $\F. Then
\begin{itemize}
\item [(i)] $\langle F_i \rangle \cap \langle G_i \rangle  =  \langle F_i \cap G_i \rangle$,
\item [(ii)] $\delta \langle F_i \cap G_i \rangle = \delta \langle F_i \rangle \cap \ \delta
 \langle G_i \rangle \in \mathcal{F}$.
\end{itemize}
 The filter generated by this base is denoted
$\prod_{i \in I}^{\mathcal F} \mathcal{F}_i$,
\medskip
and is called the  $\mathcal F$--filter for the product of filtrated spaces  $X_i$. An element of
the base can then be written as
\begin{equation*}
\prod_{i \in F^c}F_i \, \overset{{\mathcal F}}{\times} \, \prod_{i
\in F} X_i \qquad (F\in {\mathcal F}).
\end{equation*}
 Therefore, the elements of the $\mathcal F$--filter are all of the subsets in $\prod_{i \in I} X_i$ which
are supersets of some box.
\end{proof}
\subsection{Some properties}
If $ \mathcal{F}$ is the filter of the cofinite \ sets in $I$, the
corresponding $\mathcal F$--filter is called \emph{the product
filter} of the family   $\{(X_i, \mathcal{F}_i)\}_{i \in I}$ and
is denoted  $\prod_{i \in I}^{cofinitos} \mathcal{F}_i$. (Should
we call it the Tychonoff filter for the product?).

 The\emph{ box filter} is by definition the corresponding filter for the case when $\mathcal{F}$ is trivial filter of
parts of $I$.

 Since the projection functions   $p_i:\prod_{i \in I} X_i\longrightarrow  X_i$  are onto,
 given any filter ${\mathcal G}$ in the product $\prod_{i \in I} X_i$, the family $p_i(\mathcal{G})$ is a filter in $X_i$.

A natural question is: for which class of filters $\mathcal{F}$ we have that
$$p_i(\prod_{i \in I}^{\mathcal F} \mathcal{F}_i)=   \mathcal{F}_i$$
---continuity of the construction---?
\begin{propo} If  $\mathcal{F}$  is a saturated filter in $I$ then  $p_i(\prod_{i \in I}^{\mathcal F} \mathcal{F}_i)=
\mathcal{F}_i$.
\end{propo}
\begin{proof}From the definition of the filter $p_i(\prod_{i \in I}^{\mathcal F} \mathcal{F}_i)$  we have
$$p_i(\prod_{i \in F^c}F_i \, \overset{{\mathcal
F}}{\times} \, \prod_{i \in F} X_i )= F_i,$$ which implies  $p_i(\prod_{i \in I}^{\mathcal F}
\mathcal{F}_i)\subseteq  \mathcal{F}_i$.

To verify the other inclusion, given $F_i \in \mathcal{F}_i$ it must be that there is a box
$\langle F_i \rangle$ for which
 $p_i(\langle F_i \rangle)=F_i$; and for this to be the case we should be able to choose and element $F_i$
 of the filter $\mathcal{F}_i$  and
therefore it is necessary the index $i$ does not belong to some element $F\in \mathcal{F}$, which
implies  $X-\{i\}\in \mathcal{F}$; that is, $\mathcal{F}$  is saturated.
\end{proof}

In particular, since the filter of the cofinite \ sets is
saturated, we have that for the product filter its projection on
$X_i$  is $\mathcal{F}_i$. But we have much more, the product
filter is ``the best" \  filter (the smallest) that satisfies this
property.
\begin{propo} Let ${\mathcal G}$ be a filter in $\prod_{i \in F} X_i$ such that $p_i(\mathcal{G})= \mathcal{F}_i $
 for each  $i\in I$. Then
$$
p_i(\prod_{i \in I}^{\mathcal F} \mathcal{F}_i)\subseteq \mathcal{G}. $$
\end{propo}
\begin{proof} Let  $A$ be an element of the base for the filter  $\prod_{i \in I}^{\mathcal F} \mathcal{F}_i$.
There is  $F\in \mathcal{F}$  for which $A=\prod_{i \in F^c}F_i \, \overset{{\mathcal F}}{\times}
\, \prod_{i \in F} X_i$. Since $F^c$ is a finite subset, we have
$$A=\bigcap_{i\in F^c}p_i^{-1}(F_i)\in \mathcal{G}$$
given that each of the elements  $p_i^{-1}(F_i)\in \mathcal{G}$ since by hypothesis
$p_i(\mathcal{G})=\mathcal{F}_i$  which implies that there is $H\in \mathcal{G}$ con $p_i(H)=F_i$
for each $i\in F^c$  and, therefore,  $H\subseteq p_i^{-1}(F_i)$.
\end{proof}
\begin{nota} This last proposition is exactly the dual of the characterization of
the Tychonoff product topology.
\end{nota}
 Given a \F -topology in the product, by $\mathcal{V}_{\mathcal{F}}(x)$ we denote the filter of the
 neighborhoods of the point $x=(x_i)$; for each space $(X_i, \mathcal{J}_i)$  we denote by
  $\mathcal{V}(x_i)$ the filter of the
  neighborhoods of the point $x_i$. The
relation between  $\mathcal{V}_{\mathcal{F}}(x)$ and the product filter of the $\mathcal{V}(x_i)$
for the Cartesian product is given by the following proposition.
\begin{propo} Given $\{(X_i, \mathcal{J}_i)\}_{i \in I}$  a family of topological spaces, $\mathcal{F}$ a filter in
$I$ and $x=(x_i)_{i\in I} \in \prod_{i \in I} \, X_i$ a point in the product space with the
\F-topology, then
$$\mathcal{V}_{\mathcal{F}}(x)= \prod_{i \in I}^{\mathcal F}
\mathcal{V}(x_i).$$
\end{propo}
\begin{proof} Given   $V_x \in \mathcal{V}_{\mathcal{F}}(x)$  there is $F\in \mathcal{F}$ and an open set in
the base of the \F-topology for which  $$x \in \prod_{i \in F^c}U_i \, \overset{{\mathcal
F}}{\times} \, \prod_{i \in F} X_i\subseteq V_x.$$
 Since each $x_i \in U_i \in \mathcal{V}(x_i)$ we
have that the box   $\prod_{i \in F^c}U_i \, \overset{{\mathcal F}}{\times} \, \prod_{i \in F}
 X_i$ belongs to the base of the filter $\mathcal{V}_{\mathcal{F}}(x)$  and therefore
 $V_x\in \mathcal{V}_{\mathcal{F}}(x)$.
\end{proof}
\section{Product of uniform spaces}
As
a last application, let's see the construction of a uniform structure for the Cartesian product of
sets.
\subsection{Notation and preliminary concepts}
The notion of \emph{uniform space} was introduced by A. Weil [We] in 1937; basically it concerns a
set $X$ together with a family $\mathcal{U}$ of subsets of  $X\times X$ that satisfy certain
natural conditions if we keep in mind metric spaces. Since we are dealing with subsets of $X\times
X$, that is of relations in $X$, we should remember some basic definitions.

\medskip
By definition a \emph{relation} ${G}$ in $X$ is a subset  $G
\subseteq X\times X$. It is customary to write $x G y$ instead of
$(x, y) \in G$. We define the composition of two relations ${ G}$
and $H$ as $$G\circ H := \{(x,z) :  \text{ there is  } y \in X
\text{ such that }(x,y)\in G \text{ and } (y,z)\in H \}.$$ We
denote $G^{(2)}= G\circ G$.

\medskip
  The inverse $ G^{-1}$ of the relation ${\cal G}$ is
define as  $$ G^{-1}:=\{(x,z) : (z,x)\in G \}.$$ The relation $$\Delta (X)=\{(x,x):x \in X\}$$ is
called the \emph{diagonal} in $X$.

\medskip

  A \emph{uniformity} for a set $X$ is a family  $\mathcal{U}=\{U_a\}_{a\in A}$ of relations in $X$ that satisfy
the following conditions:
\begin{enumerate}
\item $\Delta(X)\subseteq U$ for every $U\in \mathcal{U}$ ---the elements of $\mathcal{U}$ are
 then called \emph{entourages} of the diagonal---.
\item If  $U\in \mathcal{U}$  then $U^{-1}\in \mathcal{U}$.
\item If $U\in \mathcal{U}$ there exists $V\in  \mathcal{U}$ such that $V\circ V \in \mathcal{U}$.
\item If $U, V\in \mathcal{U}$   then $U\cap V\in \mathcal{U}$.
\item If   $U\in \mathcal{U}$ and $U\subseteq V$ then  $V\in \mathcal{U}$.
\end{enumerate}
The pair  $(X,\mathcal{U})$  is called a \emph{uniform space}.  If
$x,y \in X$ are such that $(x,y)\in U$ for $U\in \mathcal{U}$ we
say that $x,y$ are $U$-near. Observe that a uniformity
$\mathcal{U}$ is in particular a filter in $X\times X$.

The collection $Unif(X)$ of the uniformities over a fixed set $X$ with the usual order among sets,
is a complete lattice with first and last elements $\{X \times X\}$ and $\{M\subseteq X\times X:
\Delta(X)\subseteq M\}$ respectively.

\medskip
A subfamily   $\mathcal{B}\subset \mathcal{U}$ of a uniformity $\mathcal{U}$ is a \emph{base} for
$\mathcal{U}$ if and only if each entourage  in $\mathcal{U}$ contains an element of
$\mathcal{B}$.

The following proposition gives necessary and sufficient conditions so that a family of relations
in $X$ can be a base of a unique uniformity for $X$.

\begin{propo} Let $X$ be a set and $\mathcal{B}$ a non-empty collection of subsets of
 $X\times X$. $\mathcal{B}$ is a base of a uniformity for $X$ if and only if

\begin{enumerate}
 \item Each element of $\mathcal{B}$ has a diagonal $\Delta(X)$.
 \item If  $U\in \mathcal{B}$ then $U^{-1}$ contains some element of  $\mathcal{B}$.
\item  If $U\in \mathcal{B}$ then there is $V \in \mathcal{B}$  such that  $V\circ V \subseteq U$.
\item  If   $U, V \in \mathcal{B}$ then
there is $W \in \mathcal{B}$
 such that $W\subseteq U\cap V$ ---$\mathcal{B}$  is a filter base---.
\end{enumerate}
\end{propo}
 The uniformity $\mathcal{U}$ generated by the previous base B$\mathcal{B}$, is formed by
all the supersets of the base, that is, $U\in \mathcal{U} $  if and only if there is $B\in
\mathcal{B}$  such that  $B\subseteq U$.

\medskip
A function  $f:(X,\mathcal{U})\longrightarrow (Y,\mathcal{V})$  between uniformed spaces is said to
be \emph{uniformly }continuous if for each  entourage $V \in \mathcal{V}$ there is $U \in
\mathcal{U}$ such that $(x,y)\in U$
 implies  $(f(x),f(y))\in V$ ---$f(x), f(y)$ are so close as we wish as long as  $x,y$ are also close---.

\subsection{The construction}

Let's consider the case of a family $\{ X_i\}_{i\in I }$ of non-empty sets and for each index $i$ a
uniformity $\mathcal{U}_i$ in  $X_i$. In other words, we consider a family of uniform spaces
$\{(X_i, \mathcal{U}_i)\}_{i \in I}$.

Given a Cartesian product $X=\prod_{i \in I} X_i$ of the sets $X_i$, we ask: how do we build a
uniformity in $X$ starting with the family of uniformities $\{\mathcal{U}_i\}_{i \in I}$  in the
factor sets? Since we must consider relations in $X$ we identify
 $$X\times X=\prod_{i \in I} X_i \times \prod_{i \in I} X_i = \prod_{i \in I} (X_i \times X_i) .$$
For each $i\in I$  we consider a base $\mathcal{B}_i$  for $\mathcal{U}_i$. Let's define a box
$\bold B $ in the product $X\times X$, as
\[
\bold B  := \langle B_i \rangle :=  \prod_{i\in I} B_i, \ B_i \in \mathcal{B}_i,\
 i \in I. \ \ (\text{Note that }B_i \subseteq X_i \times X_i)
\]
The set  $\delta \langle B_i \rangle$ of \emph{distinguished indexes} of a box is defined as
\[
\delta(\bold B)=\delta \langle B_i \rangle := \{i \in I \mid B_i = X_i \times X_i\}.
\]
\begin{propo}Let $ \mathcal{F}$ be a filter $I$. The boxes $\langle B_i \rangle$ such that
 $\delta \langle B_i \rangle \in \mathcal{F}$  form a base $ \mathcal{B}$  for a
uniformity in $\prod_{i \in I} X_i$.
\end{propo}
\begin{proof} We will use for the proof the base characterization  in Proposition 7.1.  An
element  $\langle B_i \rangle$ of the base $\mathcal{B}$ can be
denoted as\begin{equation*} \langle B_i \rangle = \prod_{i \in
F^c}B_i \, \overset{{\mathcal F}}{\times} \, \prod_{i \in F} (X_i
\times X_i) \qquad (F\in {\mathcal F}).
\end{equation*}
\begin{enumerate}
  \item Let $\langle B_i \rangle$ be a base
element; for each $i\in I$ and each  ${B}_i \in \mathcal{B}_i$, we have $\Delta(X_i) \subseteq
B_i$. Then
$$\Delta(X)=\prod_{i \in I}\Delta(X_i) \subseteq \prod_{i \in I}B_i$$
\item  Let  $\langle B_i \rangle$ be a base element; for each
$i\in I$ and each   ${B}_i \in \mathcal{B}_i$  there is  ${U}_i
\in \mathcal{B}_i$  with  $U_i \subseteq B_i^{-1}$. Observe that
$\prod_{i \in I}U_i \in \mathcal{B}$ and $$\prod_{i \in I}U_i
\subseteq \prod_{i \in I}B_i^{-1}= \prod_{i \in F^c}B_i^{-1}\,
 \overset{{\mathcal F}}{\times} \, \prod_{i \in F}
(X_i \times X_i) \qquad (F\in {\mathcal F}).$$
 \item  Let $\langle B_i \rangle$ be a base element; for each $i\in I$ and each  ${B}_i \in \mathcal{B}_i$,
there is  ${U}_i \in \mathcal{B}_i$ with  $U_i\circ U_i \subseteq B_i$. Then
$$\prod_{i \in I}(U_i \circ U_i)\subseteq \prod_{i \in F^c}B_i\,
 \overset{{\mathcal F}}{\times} \, \prod_{i \in F}
(X_i \times X_i) \qquad (F\in {\mathcal F}).$$
\item Let $\langle C_i \rangle$ and $\langle D_i \rangle$  base elements; for each $i\in I$ there is
$W_i \in \mathcal{B}_i$   with  $W_i \subseteq C_i\cap D_i$ ---if $C_i=D_i=X_i\times X_i$ we take
$W_i= X_i\times X_i$---. Since \F \  is a filter, we have that $\prod_{i \in I}W_i$ is in the base
and
\[\prod_{i \in I}W_i\subset \prod_{i \in I}C_i \ \bigcap \ \prod_{i \in I}D_i =\prod_{i \in I}(C_i \cap
D_i) .   \qed
\]
\end{enumerate}\renewcommand{\qed}{}
\end{proof}

 The uniformity generated by the previous base is denoted  $ \prod_{i \in I}^{\mathcal F}
\mathcal{F}_i$, and is called the $\mathcal F$--uniformity for the product of uniform spaces $X_i$.
Therefore, the elements of the $\mathcal F$--uniformity are all subsets of $X\times X$  that are
supersets of some box.

\medskip
If $ \mathcal{F}$ is the filter of the cofinite \ sets in $I$,
the corresponding  $\mathcal F$--uniformity is called the
\emph{product uniformity} of the family $\{(X_i,
\mathcal{U}_i)\}_{i \in I}$   and is denoted  $\prod_{i \in
I}^{cofinitos} \mathcal{U}_i$. (Should we call it the Tychonoff
uniformity for the product?). The product uniformity is
characterized for being the smallest uniformity for which the
projection functions  $p_i:\prod_{i \in I} X_i\longrightarrow X_i$
are continuous.

\medskip
To each uniform space  $(X,\mathcal{U})$ it is associated in natural way a topology for the set $X$
denoted $\mathcal{J}_{\mathcal{U}}$ and defined as follows:

If $U\in \mathcal{U}$ y $x\in X$ we define the \emph{U--entourage} of $x$ as
$$U[x]:=\{y: (x,y)\in U\}.$$
We define
$$\mathcal{J}_{\mathcal{U}}:= \{ G\subseteq X: \text{ for each  }x\in G \text{ there is }U\in \mathcal{U}
\text{ such that } U[x]\subseteq G\}.$$ In other words,  $G\in \mathcal{J}_{\mathcal{U}}$ if and
only if ${ G}$ contains a \emph{U--entourage}
 around each one of its points ---dual of the metric spaces and the balls---.

  Finally, we emphasize the
fact that, over a product of uniform spaces, the \F-uniformity induces the \F-topology, that is,
$$\mathcal{J}_{\mathcal{F}-uniformitu}=\mathcal{F}-\text{topology}.$$

\end{document}